\documentclass[a4, 11pt] {amsart} 
\usepackage{amssymb,times, amscd,amsmath,amsthm, xypic}
\usepackage{geometry}
\author{Laurentiu Maxim$^{(*)}$}
\address{Department of Mathematics,
University of Wisconsin-Madison,
480 Lincoln Drive, 
Madison WI 53706-1388, USA}
\email
{maxim@math.wisc.edu}

\author{Shoji Yokura$^{(**)}$}

\address{Department of Mathematics and Computer Science,
Graduate School of Science and Engineering\\ Kagoshima University,
1-21-35 Korimoto, 
Kagoshima 890-0065, Japan}
\email
{yokura@sci.kagoshima-u.ac.jp}

\date{}

\thanks{(*) Partially supported by the Romanian Ministry of National Education, CNCS-UEFISCDI, grant PN-III-P4-ID-PCE-2016-0030. \\
\quad (**) Partially supported by JSPS KAKENHI 16H03936}

\title 
[]{Homological congruence formulae \\ for characteristic classes of singular varieties}

\begin{document} 
\numberwithin{equation}{section}
\newtheorem{thm}[equation]{Theorem}
\newtheorem{pro}[equation]{Proposition}
\newtheorem{prob}[equation]{Problem}
\newtheorem{qu}[equation]{Question}
\newtheorem{cor}[equation]{Corollary}
\newtheorem{con}[equation]{Conjecture}
\newtheorem{lem}[equation]{Lemma}
\theoremstyle{definition}
\newtheorem{ex}[equation]{Example}
\newtheorem{defn}[equation]{Definition}
\newtheorem{ob}[equation]{Observation}
\newtheorem{rem}[equation]{Remark}
\renewcommand{\rmdefault}{ptm}
\def\alp{\alpha}
\def\be{\beta}
\def\jeden{1\hskip-3.5pt1}
\def\om{\omega}
\def\bigstar{\mathbf{\star}}
\def\ep{\epsilon}
\def\vep{\varepsilon}
\def\Om{\Omega}
\def\la{\lambda}
\def\La{\Lambda}
\def\si{\sigma}
\def\Si{\Sigma}
\def\Cal{\mathcal}
\def\m {\mathcal}
\def\ga{\gamma}
\def\Ga{\Gamma}
\def\de{\delta}
\def\De{\Delta}
\def\bF{\mathbb{F}}
\def\bH{\mathbb H}
\def\bPH{\mathbb {PH}}
\def \bB{\mathbb B}
\def \bA{\mathbb A}
\def \bC{\mathbb C}
\def \bOB{\mathbb {OB}}
\def \bM{\mathbb M}
\def \bOM{\mathbb {OM}}
\def \mA{\mathcal A}
\def \mB{\mathcal B}
\def \mC{\mathcal C}
\def \mR{\mathcal R}
\def \mH{\mathcal H}
\def \mM{\mathcal M}
\def \mM{\mathcal M}
\def \mT{\mathcal {T}}
\def \mAB{\mathcal {AB}}
\def \bK{\mathbb K}
\def \bG{\mathbf G}
\def \bL{\mathbb L}
\def\bN{\mathbb N}
\def\bR{\mathbb R}
\def\bP{\mathbb P}
\def\bZ{\mathbb Z}
\def\bC{\mathbb  C}
\def \bQ{\mathbb Q}
\def \mh{\mbox{MHM}}
\newcommand{\bb}[1]{\mbox{$\mathbb{#1}$}}

\def\op{\operatorname}

\maketitle

\begin{abstract} For a pair $(f, g)$ of morphisms $f:X \to Z$ and $g:Y \to Z$ of (possibly singular) complex algebraic varieties $X,Y,Z$, we present congruence formulae for the difference $f_*T_{y*}(X) -g_*T_{y*}(Y)$ of pushforwards of the corresponding motivic Hirzebruch classes $T_{y*}$. If we consider the special pair of a fiber bundle $F \hookrightarrow E \to B$ and the projection $pr_2:F \times B \to B$ as such a pair $(f,g)$, then we get a congruence formula for the difference $f_*T_{y*}(E) -\chi_y(F)T_{y*}(B)$, which at degree level yields a congruence formula for $\chi_y(E) -\chi_y(F)\chi_y(B)$, expressed in terms of the Euler--Poincar\'e characteristic, Todd genus and signature in the case when $F, E, B$ are non-singular and compact. We also extend the finer congruence identities of Rovi--Yokura to the singular complex projective situation, by using the corresponding intersection (co)homology invariants.
\end{abstract}


\section{Introduction}

It is well-known that the Euler--Poincar\'e characteristic is multiplicative, i.e., $\chi(X \times Y) =\chi(X)\chi(Y)$. Moreover, $\chi$ is multiplicative for \emph{any} topological fiber bundle, i.e., if $F \hookrightarrow E \to B$ is a (topological) fiber bundle with total space $E$,  fiber space $F$ and base space $B$, then $\chi(E) = \chi(F)\chi(B)$. 

The signature of closed oriented manifolds is also multiplicative, i.e.,  $\sigma(X \times Y)=\sigma(X)\sigma(Y)$, but unlike the Euler-Poincar\'e characteristic the signature  is not necessarily multiplicative for fiber bundles, unless certain conditions are satisfied. For example, S. S. Chern, F. Hirzebruch and J.-P. Serre \cite{CHS} proved the following result:
\begin{thm}[Chern--Hirzebruch--Serre]\label{chs} Let $F \hookrightarrow E \to B$ be a fiber bundle of closed oriented manifolds. If the fundamental group $\pi_1(B)$ of the base space $B$ acts trivially on the cohomology group $H^*(F; \mathbb R)$ of the fiber space $F$ (e.g., if $B$ is simply-connected), then 
$\sigma(E) =\sigma(F)\sigma(B).$
\end{thm}
M. Atiyah \cite{At}, F. Hirzebruch \cite {Hir0} and K. Kodaira \cite{Ko} independently gave examples of differentiable fibre bundles $F \hookrightarrow E \to B$ for which $\sigma(E) \not =\sigma(F)\sigma(B)$. These examples are all of real dimension $4$, with $\sigma(E) \not = 0$. However, $\op{dim}_{\mathbb R} F = \op{dim}_{\mathbb R}B =2$, so $\sigma(F) = \sigma(B) =0$ by the definition of the signature (which is defined to be zero if the real dimension of the manifold is not divisible by $4$). Moreover, M. Atiyah \cite{At} and W. Meyer \cite{Me72} obtained formulae for the monodromy contribution in the deviation of such ``multiplicativity results''. In particular, Theorem \ref{chs} is already true for an even dimensional fiber $F$, if $\pi_1(B)$ acts trivially on the middle dimensional cohomology
 $H^{\dim(F)/2}(F,\mathbb{R})$ of the fiber.
H. Endo \cite{Endo}  and W. Meyer \cite{Mey}  studied further surface bundles over surfaces, and they showed that the signature of such fiber bundles is always divisible by $4$, i.e., $\sigma(E) \equiv 0 \op{mod} 4.$ In particular, for such surface bundles over surfaces one also has that $\sigma(E)\equiv \sigma(F)\sigma(B) \, \, \op{mod} 4$.

More generally, I. Hambleton, A. Korzeniewski and A. Ranicki  \cite{HKR}  generalized the latter observation to $PL$-bundles of any dimension, and proved the following:
\begin{thm}[Hambleton--Korzeniewski--Ranicki]
For a $PL$ fiber bundle $F \hookrightarrow E \to B$ of closed, connected, compatibly oriented $PL$ manifolds, one has: $\sigma(E) \equiv \sigma(F) \sigma(B) \op{mod} 4.$
\end{thm}

In \cite{CMS} (also see \cite {CLMS1}, \cite{CLMS2}, \cite{MS2}) S. Cappell, L. Maxim and J. Shaneson have extended  Theorem \ref{chs} to the Hirzebruch $\chi_y$-genus (see \S2 for a definition) in the context of a ``complex algebraic fiber bundle" $F \hookrightarrow E \to B$, by which is meant an algebraic morphism $\pi: E \to B$ of compact complex algebraic varieties, which is also a topological fiber bundle, such that the compact complex variety $F$ is isomorphic to a fiber of $\pi$  over every connected component of $B$.
\begin{thm}[Cappell--Maxim--Shaneson]\label{CMS} Let $F \hookrightarrow E \to B$ be a smooth complex algebraic fiber bundle (thus $F,E,B$ are smooth and compact). If the fundamental group $\pi_1(B)$ of the base space $B$ acts trivially on the cohomology $H^*(F; \mathbb Q)$ of the fiber space $F$, then the Hirzebruch $\chi_y$-genus is multiplicative, i.e., $\chi_y(E) = \chi_y(F)\chi_y(B)$.
\end{thm}

The Hirzebruch $\chi_y$-genus $\chi_y(X) \in \bZ[y]$ of a compact complex algebraic manifold $X$ was introduced by F. Hirzebruch \cite{Hirzebruch} (also see \cite{HBJ}) in order to generalize his famous Hirzebruch-Riemann-Roch theorem. For the distinguished values $y=-1, 0, 1$ of the parameter $y$, the Hirzebruch $\chi_y$-genus specializes to: 
\begin{itemize}
\item $\chi_{-1}(X) =\chi(X)$ the (topological) \emph{Euler-Poincar\'e characteristic}, 
\item $\chi_0(X) = \tau(X)$ the \emph{Todd genus} (alias holomorphic Euler characteristic $\chi(X, \mathcal O_X)$) and
\item  $\chi_1(X)=\sigma(X)$ the \emph{signature}.
\end{itemize}
Thus the Hirzebruch $\chi_y$-genus unifies these three important characteristic numbers. 

In fact, \cite[Proposition 2.3]{CMS} deals with multiplicativity properties of the more general Hirzebruch $\chi_y$-genus (as recalled in \S3) of possibly singular complex algebraic varieties, which for a compact variety $X$ is the degree $\int_X {T_y}_*(X)$ of the motivic Hirzeburch class ${T_y}_*(X)\in H_*(X) \otimes \bQ[y]$ (as recalled in \S3) introduced in \cite{BSY1} (cf. also \cite{BSY2}, \cite{Sch}, \cite{YokuraMSRI}, \cite{BSY3}).  
Moreover, the following characteristic class generalization of Theorem \ref{CMS} is proved in \cite[Corollary 4.11]{CMS} (see also \cite{CLMS1,CLMS2}):
\begin{thm}[Cappell--Maxim--Shaneson]\label{CMS2}  Let $f:E \to B$  be a proper algebraic map of complex algebraic varieties, with $B$ smooth and connected, so that all direct image sheaves $R^jf_*\mathbb Q_E$
are locally constant (e.g., $f$ is a locally trivial topological fibration). Let $F$ be the general fiber of $f$, and assume that $\pi_1(B)$ acts trivially on the 
cohomology of $F$ (e.g. $\pi_1(B)=0$), i.e., all these $R^jf_* \mathbb Q_E$ 
are constant. Then
\begin{equation*}
f_*T_{y*}(E) = \chi_y(F)T_{y*}(B). 
\end{equation*}
\end{thm}
More general ``stratified multiplicative properties'' for an arbitrary proper algebraic morphism $f:E\to B$, describing the difference $f_*T_{y*}(E) - \chi_y(F)T_{y*}(B)$ in terms of corresponding invariants of strata of $f$ are obtained in \cite{CMS}, \cite{CLMS1}, \cite{CLMS2}, \cite{MS1}, etc.

On the other hand, in \cite{BSY3} J.-P. Brasselet, J. Sch\"urmann and S. Yokura give some explicit description of the motivic Hirzebruch classes in terms of other known homology classes, and in \cite{Yo2} Yokura  expresses the difference $\chi_y(E) - \chi_y(F)\chi_y(B)$ for a smooth complex algebraic fiber bundle $F \hookrightarrow E \to B$ in terms of the Euler-Poincar\'e characteristic, Todd genus and signature. As a byproduct of such explicit computations, S. Yokura derives the congruence $\sigma(E)\equiv \sigma(F)\sigma(B) \, \, \op{mod} 4$, thus reproving the above mentioned result of Hambleton--Korzeniewski--Ranicki \cite{HKR} in the complex algebraic context. 
Yokura also shows that \emph{the Euler-Poincar\'e characteristic is the only multiplicative specialization of the $\chi_y$-genus for such fiber bundles}.
Furthermore, motivated by the proof of the above congruence formula, in \cite[Theorem 3.1 and Theorem 4.1]{RY} C. Rovi and S. Yokura prove the following extension to the Hirzebruch $\chi_y$-genus of the congruence formula for the signature modulo $4$ and that for the signature modulo $8$ due to  C. Rovi \cite{Rov, Rov2}:
\begin{thm}[Congruence formulae for Hirzebruch $\chi_y$-genera $\op{mod} \, \, 4$ and $8$] \label{mod8} For a complex algebraic fiber bundle $F \hookrightarrow E \to B$ (thus $F,E,B$ are smooth and compact), one has the following congruence formulae:
\begin{enumerate}
\item For any odd integer $y$, $\chi_y(E) \equiv \chi_y(F)\chi_y(B) \, \op{mod} \, \, 4$.
\item If $y \equiv 3 \, \op{mod} \, 4$, then $\chi_y(E) \equiv \chi_y(F)\chi_y(B) \, \op{mod} \, 8.$
\item If $y \equiv 1 \, \op{mod} \, 4$, then $\chi_y(E) \equiv \chi_y(F)\chi_y(B) \, \op{mod} \, 8 \Longleftrightarrow \, \sigma(E) \equiv \sigma(F)\sigma(B) \, \op{mod} \, 8.$ 
\end{enumerate}
\end{thm}

In this paper, we consider congruence identities for the motivic Hirzebruch classes $T_{y*}(X)$ and their intersection homology counterpart.
We prove the following result:
\begin{thm}\label{thm0} Let $f:X \to Z$ and $g:Y \to Z$ be morphisms of complex algebraic varieties with the same target $Z$, and assume that $X$, $Y$ and $Z$ are compact. Then we have the following congruences for the difference $f_*T_{y*}(X) - g_*T_{y*}(Y)$:
\begin{align*}
(1) \,& f_*T_{y*}(X) - g_*T_{y*}(Y) \equiv \frac{f_*c_*(X)\otimes \mathbb Q -g_*c_*(Y)\otimes \mathbb Q}{2}(1-y)+ \\
& \hspace{4cm} \frac{f_*L_*^H(X) -g_*L_*^H(Y)}{2}(1+y) \,\op{mod} \Bigl (H_{*}(Z) \otimes \mathbb Q[y] \Bigr)(1-y^2).
\end{align*}
\begin{align*}
(2)\,  & f_*T_{y*}(X) - g_*T_{y*}(Y) \equiv \frac{f_*c_*(X)\otimes \mathbb Q -g_*c_*(Y)\otimes \mathbb Q}{2}(y^2 -y) 
\\
& \hspace{4cm} +\Bigl( f_*td_*^H(X)-g_*td_*^H(Y) \Bigr)(1-y^2) 
+ \frac{f_*L_*^H(X) -g_*L_*^H(Y)}{2}(y+y^2) \\
& \hspace{9cm} \op{mod} \Bigl (H_{*}(Z) \otimes \mathbb Q[y] \Bigr)(y-y^3).
\end{align*}
Here, $c_*(X)\otimes \bQ=T_{-1*}(X)$ is the rationalized Chern--Schwartz--MacPherson class, $td^H_*(X):=T_{0*}(X)$ and $L^H_*(X):=T_{1*}(X)$.

In particular, by taking the degree in the above formulae, we have:
$$ (3) \, \,  \chi_y(X) - \chi_y(Y) \equiv  \frac{\chi(X) -\chi(Y)}{2}(1-y) + \frac{\sigma^H(X) -\sigma^H(Y)}{2}(1+y)  \, \, \op{mod} \, 1-y^2, \qquad \qquad $$
\begin{align*}
(4) \, \, \chi_y(X) - \chi_y(Y) \equiv & \frac{\chi(X) -\chi(Y)}{2}(y^2 -y) + \left (\tau^H(X)-\tau^H(Y) \right )(1-y^2)  \qquad \qquad \qquad \qquad \qquad  \\
&   \hspace{4cm}  +  \frac{\sigma^H(X) -\sigma^H(Y)}{2}(y+y^2)  \,\, \, \op{mod} \, y-y^3 ,
\end{align*}
where for a (possible singular) complex algebraic variety $X$ we let $\tau^H(X)=\chi_0(X)$ and $\sigma^H(X)=\chi_1(X)$.
\end{thm}

\begin{rem} Note that formulae (1) and (2) of Theorem \ref{thm0} also hold for non-compact spaces, provided that $f$ and $g$ are proper morphisms.
 Here, $ \op{mod}\Bigl(H_*(Z)\otimes \bQ[y] \Bigr)(1-y^2)$ in (1) means specializing $y^2=1$ in
 $$\begin{CD} 
    H_*(Z)\otimes \bQ[y]/(1-y^2) @> y^2=1 > \simeq > H_*(Z)\otimes (\bQ\oplus \bQ\cdot y),
   \end{CD}$$
   and $ \op{mod} \Bigl(H_*(X)\otimes \bQ[y]\Bigr)(y-y^3)$ in (2) means specializing $y^3=y$ in
 $$\begin{CD} 
    H_*(Z)\otimes \bQ[y]/(y-y^3) @> y^3=y > \simeq > H_*(Z)\otimes (\bQ\oplus \bQ\cdot y \oplus \bQ\cdot y^2)\:.
   \end{CD}$$
The $\chi_y$-genus in (3) and (4) already lives in $\bZ[y]\subset \bQ[y]$, with $\bZ[y]/(1-y^2)\simeq \bZ\oplus \bZ\cdot y$ and
$\bZ[y]/(y-y^3)\simeq \bZ\oplus \bZ\cdot y \oplus \bZ\cdot y^2$. 
\end{rem}


In fact, it suffices to prove Theorem \ref{thm0} just for one morphism (i.e., with either $X$ or $Y$ empty). But the formulation with two morphisms allows us to compare a given morphism $f: E\to B$ with the second factor projection $g=pr_2: F\times B\to B$
of a product with fiber $F$. Then the Chern class, resp., Euler characteristic contributions in Theorem \ref{thm0} cancel out if we assume that all 
(compact) fibers of $f: E\to B$ have the same Euler characteristic as $F$, i.e., in terms of (algebraically) constructible functions: $f_*(\jeden_E)=\chi(F)\cdot \jeden_B$.
This assumption is much weaker than the topological assumption that $f: E\to B$ is a ``complex algebraic fiber bundle'' with fiber $F$,
or the cohomological assumption that all direct image sheaves $R^jf_*\bQ_E$ are locally constant, with $F$ isomorphic to a fiber of $f$ over every connected component of $B$.
\begin{cor}\label{cor} Let $f:E \to B$ be a proper morphism, and assume that all the fibers of $f$ have the same Euler characteristic $\chi(F)$ as the compact algebraic variety $F$.
Then we have 
\begin{align*}
(1) \, \,  f_*T_{y*}(E) -& \chi_y(F)T_{y*}(B) \\
& \equiv \frac{f_*L_*^H(E) -\sigma^H(F)L_*^H(B)}{2} (1+y)  \,\, \op{mod} \Bigl (H_{*}(B) \otimes \mathbb Q[y] \Bigr)(1-y^2).
\end{align*}
If $B$ (and therefore also $E$) is compact, then by taking the degree we have:
$$\chi_y(E) -\chi_y(F)\chi_y(B) \equiv \frac{\sigma^H(E)-\sigma^H(F)\sigma^H(B)}{2}(1+y) \,\, \op{mod} \, \, 1-y^2.$$
\begin{align*}
(2) \, \, f_*T_{y*}(E) -\chi_y(F)T_{y*}(B) & \equiv \Bigl (f_*td_*^H(E) -\tau^H(F)td_*^H(B) \Bigr)(1-y^2) \\
& \hspace{1cm}+ \frac{f_*L_*^H(E)-\sigma^H(F)L_*^H(B)}{2}(y+ y^2) \\
& \hspace{4cm}\op{mod} \Bigl (H_{*}(B) \otimes \mathbb Q[y] \Bigr)(y-y^3)
\end{align*}
If $B$ (and therefore also $E$) is compact, then by taking the degree we have:
\begin{align*}
\chi_y(E) -\chi_y(F)\chi_y(B) & \equiv \Bigl (\tau^H(E) -\tau^H(F)\tau^H(B) \Bigr)(1-y^2)\\
& \hspace{2cm} + \frac{\sigma^H(E)-\sigma^H(F)\sigma^H(B)}{2}(y^2+y) \,\, \op{mod} \, \, y-y^3.
\end{align*}
\end{cor}
\begin{rem} The degree part (i.e., the integral polynomial part) of Corollary \ref{cor} (1) plays an essential role in the proof of Theorem \ref{mod8}. 
\end{rem}

Intersection homology flavoured analogues of the above results are  discussed in \S 6. For instance, we prove the following intersection (co)homology counterpart of Theorem \ref{mod8}: 
\begin{thm}\label{ihmod8}
 Let $F\hookrightarrow E \to B$ be a complex algebraic fiber bundle of pure-dimensional complex projective algebraic varieties. Then
 \begin{enumerate}
  \item For any odd integer $y$, $I\chi_y(E)\equiv I\chi_y(F)I\chi_y(B) \op{mod} 4$.
  \item If $y \equiv\; 3 \;mod \;4$, then $I\chi_y(E)\equiv I\chi_y(F)I\chi_y(B) \op{mod} 8$.
  \item If $y \equiv\; 1 \;mod \;4$, then $I\chi_y(E)\equiv I\chi_y(F)I\chi_y(B) \op{mod} 8 \Leftrightarrow 
  \sigma(E) \equiv \sigma(F)\sigma(B) \op{mod} 8 $.
 \end{enumerate}
\end{thm}
Here, $\sigma(-)$ denotes the Goresky--MacPherson intersection (co)homology signature, and $I\chi_y(-)$ is the intersection (co)homology counterpart of the $\chi_y$-genus (defined by using the Hodge structure on the intersection cohomology groups of complex projective varieties). In particular, Theorem \ref{ihmod8} extends Theorem \ref{mod8} from smooth algebraic fiber bundles to the case of algebraic fiber bundles of pure-dimensional complex projective algebraic varieties, which are only rational homology manifolds.

\medskip

The paper is organized as follows. In \S2 we recall the definition of the Hirzebruch $\chi_y$-genus and of cohomology Hirzebruch classes in the smooth context. Extensions of these notions to the singular context are discusses in \S3, where a short overview of the theory of  motivic Hirzebruch classes is given. Theorem \ref{thm0} is proved in \S4, whereas applications to complex algebraic fiber bundles (Corollary \ref{cor}) are discussed in \S5. Finally, intersection Hirzebruch classes are introduced in \S 6, and congruence identities for such classes are discussed, e.g., see Theorem \ref{thm6}. Theorem \ref{ihmod8} is also proved in this final section.

\section{Hirzebruch $\chi_y$-genera and Hirzebruch classes $T_y$}

\begin{defn}[Hirzebruch $\chi_y$-genus]
For a compact complex algebraic manifold $X$ the Hirzebruch \emph{$\chi_{y}$-genus} $\chi_y(X)$ is  defined by 
$$\chi_{y}(X):= 
\sum_{p\geq 0} \chi(X,\Lambda^{p}T^{*}X)y^p
= \sum_{p\geq 0} \Biggl( \sum_{i\geq 0}
(-1)^{i}\op{dim}_{\mathbb C}H^{i}(X,\Lambda^{p}T^{*}X) \Biggr)y^p\:.
$$
\end{defn}
\begin{rem}
Let $$\chi^p(X):= \chi(X,\Lambda^{p}T^{*}X)$$ be the Euler characteristic of the sheaf $\Lambda^{p}T^{*}X$. Then the Hirzebruch $\chi_y$-genus is the generating function for $\chi^p(X)$, i.e., 
$$\chi_{y}(X)= \sum_{p\geq 0} \chi^p(X)y^p.$$
\end{rem}

The Hirzebruch $\chi_{y}$-genus has the following important properties:
\begin{enumerate}
\item  Since $\Lambda^{p}T^{*}X =0$ for $p> \op{dim}_{\mathbb C}X$, $\chi_{y}(X)$ is a polynomial of degree at most $\op{dim}_{\mathbb C} X$. 
\item For the three distinguished values $-1, 0, 1$ of $y$, we have the following:
\begin{enumerate}
\item $\chi_{-1}(X)=\chi(X)$ 
is the Euler--Poincar\'e characteristic. 
\item $ \chi_0(X) = \chi^0(X)=\tau (X)$ is the Todd genus. 
\item $\chi_1(X)=\sigma(X)$ 
is the signature. For a projective complex algebraic (or compact K\"ahler) manifold, this identity follows from the Hodge index theorem as in \cite[Theorem 15.8.2, p.125]{Hirzebruch}. For a compact complex (algebraic) manifold it is a consequence of the Atiyah-Singer index theorem \cite{AS1, AS2, AS3} as explained, e.g., in \cite[Appendix I, Chapter 25, p.190]{Hirzebruch}.
\end{enumerate}
\item $\chi_y$ is multiplicative, i.e., $\chi_y(X \times Y) = \chi_y(X) \chi_y(Y)$. 
\end{enumerate}

The following \emph{duality formula}, which follows from Serre duality \cite[a special case of (14) on p.123]{Hirzebruch} (and also see \cite{Kotschick1, Kotschick2, KS}), plays a key role in proving the congruence identities of Theorem \ref{mod8} for the signature modulo $4$ and $8$ (see also \S 6 for the corresponding intersection (co)homology results):
\begin{thm}\label{duality} For a compact complex algebraic manifold $X$ of pure complex dimension $n$ we have
$$\chi^p(X) = (-1)^n \chi^{n-p}(X).$$
\end{thm}

More generally, one can make the following
\begin{defn}[Hirzebruch $\chi_y$-genus of a vector bundle] For $E$ a holomorphic vector bundle over $X$, the Hirzebruch $\chi_y$-genus of $E$ is defined by
$$\chi_{y}(X,E):= 
\sum_{p\geq 0} \chi(X,E\otimes \Lambda^{p}T^{*}X) y^{p}\\
=\sum_{p\geq 0} \Biggl( \sum_{i\geq 0}
(-1)^{i}\op{dim}_{\mathbb C}H^{i}(X,E\otimes \Lambda^{p}T^{*}X) \Biggr) y^{p} \:.
$$
\end{defn}
\begin{thm}[The generalized Hirzebruch--Riemann--Roch theorem]
\begin{equation} \label{eq:gHRR} 
\chi_{y}(X,E)= \int_{X} T_{y}(TX)\cdot ch_{(1+y)}(E) \cap [X]
\quad \in \bb{Q}[y], \tag{gHRR}
\end{equation}
with $$ch_{(1+y)}(E):= \sum_{j=1}^{rank\;E} e^{\beta_{j}(1+y)}
\quad \text{and} \quad T_{y}(TX):= \prod_{i=1}^{\op{dim} X} \Biggl (\frac{\alpha_i(1+y)}{1-e^{-\alpha_i(1+y)}} -\alpha_i y \Biggr),$$
where $\beta_{j}$ are the Chern roots of $E$, and $\alpha_{i}$ are the Chern
roots of the tangent bundle $TX$.
\end{thm}

\begin{defn} $T_{y}(TX)$ is called the \emph{cohomology Hirzebruch class} of $X$.
\end{defn}

The special case of (\ref{eq:gHRR}) when $y=0$ is the famous \emph{Hirzebruch--Riemann--Roch theorem}:
\begin{thm}
\begin{equation} \label{eq:HRR} 
\chi(X,E)= \int_{X} td(TX)\cdot ch(E) \cap [X]. \tag{HRR}
\end{equation}
Here $\displaystyle ch(E)= \sum_{j=1}^{rank\;E} e^{\beta_{j}}$ is the Chern character and
$\displaystyle td(TX)= \prod_{i=1}^{\op{dim} X} \Bigl (\frac{\alpha_i}{1-e^{-\alpha_i}}\Bigr)$ is the Todd class.
\end{thm}

\begin{rem}\label{spe} We note that the normalized
power series 
$\displaystyle \frac{\alpha(1+y)}{1-e^{-\alpha(1+y)}} -\alpha y$
 specializes to
\begin{enumerate}
\item $(y=-1)$: $1+\alpha$ 
\item $(y=\, \, \, \, 0)$: $\frac{\alpha}{1-e^{-\alpha}}$ 
\item $(y= \, \, \, \, 1)$: $\frac{\alpha}{\tanh \alpha}$ 
\end{enumerate}
Therefore the cohomology Hirzebruch class $T_{y}(TX)$ unifies the
following three distinguished and important cohomology characteristic classes of $TX$:
\begin{enumerate}
\item $(y=-1)$: $c(TX) = \prod_{i=1}^{\op{dim} X}  (1+\alpha)$ the total Chern class 
\item $(y=\, \, \, \, 0)$: $td(TX) = \prod_{i=1}^{\op{dim} X} \frac{\alpha}{1-e^{-\alpha}}$ the total Todd class 
\item $(y= \, \, \, \, 1)$: $ L(TX) = \prod_{i=1}^{\op{dim} X}  \frac{\alpha}{\tanh \alpha}$ the total Thom--Hirzebruch $L$-class. 
\end{enumerate}
\end{rem}
\section{Hirzebruch $\chi_y$-genus and motivic Hirzebruch class $T_{y*}$ for singular varieties}

The Hirzebruch $\chi_y$-genus can be extended to the case of singular varieties, by using Deligne's mixed Hodge structures.

\begin{defn}
The \emph{Hodge--Deligne polynomial} of a complex algebraic variety $X$ is defined by:
$$\chi _{u,v}(X) := \sum_{i, p, q \geq 0} (-1)^i (-1)^{p+q}\op{dim}_{\bC} (Gr^p_F Gr^W_{p+q} H^i_c(X, \bC)) u^p v^q,$$
where $(W^{\bullet},F_{\bullet})$ denote the weight and resp. Hodge filtration on $H^*_c(X,\bQ)$.
\end{defn}
The Hodge--Deligne polynomial $\chi_{u,v}$ satisfies the following properties:
\begin{enumerate}
\item
$ X\cong X'$ (isomorphism) $\Longrightarrow$ $\chi _{u,v}(X) = \chi _{u,v}(X')$,

\item
$\chi _{u,v}(X) = \chi _{u,v}(X\setminus Y) + \chi _{u,v}(Y)$ for a closed subvariety $Y \subset X$

\item
$\chi _{u,v}(X \times Y) = \chi _{u,v}(X)\cdot \chi _{u,v}(Y)$,

\item
$\chi _{u,v}(pt) = 1.$ 
\end{enumerate} 
The isomorphism class of a variety $X$ shall be denoted by $[X]$, and the free abelian group generated by the isomorphism classes of varieties shall be denoted by $\op{Iso}(\mathcal V)$. Then the above property (1) implies that the following homomorphism
$$\chi_{u,v}:\op{Iso}(\mathcal V) \to \mathbb Z[u,v], \quad \chi_{u,v}([X]):=\chi_{u,v}(X)$$
is well-defined. The Grothendieck group of varieties, $K_0(\mathcal V)$, is the quotient of $\op{Iso}(\mathcal V)$ by the subgroup generated by elements of the form $[X] - [Y] - [X \setminus Y]$ with $Y$ a closed subvariety of $X$, i.e., 
$$K_0(\mathcal V):= \op{Iso}(\mathcal V)/\{[X] - [Y] - [X \setminus Y]\}.$$
An element of $K_0(\mathcal V)$ is simply denoted by $[X]$, for the equivalence class of $[X]$. Then the above property (2) implies that the homomorphism $\chi_{u,v}:\op{Iso}(\mathcal V) \to \mathbb Z[u,v]$ induces the following finer homomorphism
$$\chi_{u,v}:K_0(\mathcal V) \to \mathbb Z[u,v].$$
In \cite{Looijenga}, E. Looijenga defines the relative Grothendieck group $K_0(\mathcal V/X)$ analogously, as the free abelian group generated by the isomorphism classes $[V \xrightarrow h X]$ of  morphisms $h:V \to X$, modulo the relation 
 $$[V \xrightarrow h X] = [W \xrightarrow {h|_W} X] + [V \setminus W \xrightarrow {h|_{V \setminus W}} X]$$
  for a closed subvariety $W \subset V$.
Note that $K_0(\mathcal V/pt) = K_0(\mathcal V)$, and for a morphism $f:X \to Y$ the pushforward $f_*:K_0(\mathcal V/X) \to K_0(\mathcal V/Y)$ is defined by $f_*([V \xrightarrow h X]):=[V \xrightarrow {f \circ h} Y]$. 
So $K_0(\mathcal V/-)$ is a covariant functor. 

A natural question concerns the existence of a \emph{Grothendieck--Riemann--Roch type theorem for $\chi_{u,v}:K_0(\mathcal V) \to \mathbb Z[u,v]$}, i.e., the existence of a natural transformation
$$\mu: K_0(\mathcal V/-) \to H_*(-)\otimes \mathbb Z[u,v]$$
(with $H_*(-)$ denoting the even-degree Borel-Moore homology) 
such that for the map $a_X: X \to pt$ to a point the following diagram commutes:
$$\CD
K_0(\mathcal V/X) @> \mu >> H_*(X)\otimes \mathbb Z[u,v] \\
@V (a_X)_* VV @VV (a_X)_* V\\
K_0(\mathcal V) @>> \chi_{u,v} > \mathbb Z[u,v].\endCD
$$
It is also natural to require that such a transformation $\mu$ satisfies the following ``smooth condition" (or normalization): there exists a multiplicative cohomology characteristic  class $c\ell$ such that if $X$ is smooth then
$$\mu([X \xrightarrow {\op{id}_X} X]= c\ell(TX) \cap [X].$$
As shown in \cite[Example 5.1]{BSY1}, this ``smooth condition" implies that $(u+1)(v+1)=0$, i.e., $u=-1$ or $v=-1$, therefore we can address the above question only for $\chi_{u, -1}(X)$ or $\chi_{-1,v}(X)$. Since $\chi_{u,v}(X)$ is in fact symmetric with respect to the variables $(u,v)$, 
we can restrict our attention (after changing $u$ to $y$) to 
\begin{align*}
\chi_y(X) :=\chi_{y,-1}(X) & = \sum_{i, p \geq 0} (-1)^i \op{dim}_{\mathbb C} Gr^p_{F} \left( H^i_c (X, \mathbb C) \right) (-y)^p\\
& = \sum_{p \geq 0} \left (\sum_{i \geq 0} (-1)^{i+p} \op{dim}_{\mathbb C} 
Gr^p_{F} \left ( H^i_c (X, \mathbb C) \right )\right)  y^p.
\end{align*}
When $X$ is non-singular and compact, the purity of cohomology implies that $\chi_y(X)$ coincides with the Hirzebruch $\chi_y$-genus of $X$, which explains why it is also referred to as the Hirzebruch $\chi_y$-genus (though other terminology, such as Hodge polynomial or Hirzebruch polynomial, is also used). 
Thus, for a possibly singular variety $X$, the coefficient $\chi^p(X)$ of the above Hirzebruch $\chi_y$-genus  $\chi_y(X)$ is
$$\chi^p(X) = \sum_{i \geq 0} (-1)^{i+p} \op{dim}_{\mathbb C} 
Gr^p_{F} \left ( H^i_c(X, \mathbb C) \right ).$$
Here we remark that the degree of the above integral polynomial $\chi_y(X)$ of a possibly singular variety $X$ is at most the dimension of $X$, just like in the smooth case (cf. \cite[Corollary 3.1(1)]{BSY1}).
If we now consider the commutative diagram:
$$\CD
K_0(\mathcal V/X) @> \mu >> H_*(X)\otimes \mathbb Z[y] \\
@V (a_X)_* VV @VV (a_X)_* V\\
K_0(\mathcal V) @>> \chi_y > \mathbb Z[y] \endCD
$$
 for $[\mathbb P^n \xrightarrow {\op{id}_{\mathbb P^n}} \mathbb P^n]$, we have
$$\chi_y(\mathbb P^n) = \int_{\mathbb P^n} c\ell(T\mathbb P^n) \cap [\mathbb P^n].$$
Since $\chi_y(\mathbb P^n)=1 -y+y^2+ \cdots + (-1)^ny^n$, we must have
$$\int_{\mathbb P^n} c\ell(T\mathbb P^n) \cap [\mathbb P^n] = 1 -y+y^2+ \cdots + (-1)^ny^n.$$
In \cite{Hirzebruch}, Hirzebruch proved that such a characteristic class $c\ell$ has to be $T_y$, i.e., the cohomology Hirzebruch class. 

In \cite{BSY1} (see also \cite {BSY2}, \cite{SY}, \cite{Sch} and \cite{YokuraMSRI}), Saito's theory of mixed Hodge modules \cite{Saito} is used to give a positive answer to the above question on the existence of a Grothendieck--Riemann--Roch type theorem for $\chi_y:K_0(\mathcal V) \to \mathbb Z[y]$, namely:
\begin{thm}[Brasselet--Sch\"urmann--Yokura] There exists a unique natural transformation 
$${T_y}_*: K_0(\m V/X) \to H_*(X)\otimes \bQ[y],$$
such that
\begin{enumerate}
\item For a nonsingular variety $X$, ${T_y}_*([X \xrightarrow {\op{id}_X} X]) = T_y(TX) \cap [X]$, where $T_y(TX)$ is the cohomology Hirzebruch class of $X$.
\item For a point $X =pt$, ${T_y}_*: K_0(\m V/pt) =K_0(\mathcal V) \to H_*(pt)\otimes \bQ[y] =\bQ[y]$
is equal to the homomorphism $\chi_y:K_0(\mathcal V) \to \mathbb Z[y]$ followed by the inclusion $\mathbb Z[y] \hookrightarrow \bQ[y]$.
\end{enumerate}
The natural transformation ${T_y}_*: K_0(\m V/X) \to H_*(X)\otimes \bQ[y]$ is called the motivic Hirzebruch class transformation.
\end{thm} 
\begin{defn} For any (possibly singular) variety $X$, $$T_{y*}(X):=T_{y*}([X \xrightarrow {\op{id}_X} X])$$
is called the \emph{motivic Hirzebruch class} of $X$. 
\end{defn}

\begin{rem} As shown in \cite{BSY1},  the motivic Hirzebruch class transformation ${T_y}_*$ is functorial for proper morphisms. In particular, if $X$ is compact, then the degree of the motivic Hirzebruch class $T_{y*}(X)$ is exactly the Hirzebruch $\chi_y$-genus of $X$, i.e.,
\begin{equation}\label{deg} \int_X T_{y*}(X):=(a_X)_* T_{y*}(X)= \chi_y(X).\end{equation}
\end{rem}

Note that by specializing the natural transformation ${T_y}_*: K_0(\m V/X) \to H_*(X)\otimes \bQ[y]$ for the three distinguished values of $y=-1, 0,1$, we get by Remark \ref{spe} the following:
\begin{enumerate}
\item ($y=-1$): There exists a unique natural transformation 
$${T_{-1}}_*: K_0(\m V/X) \to H_*(X)\otimes \bQ,$$
such that for a nonsingular variety $X$, ${T_{-1}}_*([X \xrightarrow {\op{id}_X} X]) = c(TX) \cap [X].$
\item ($y= \, \, \, \, \, 0$): There exists a unique natural transformation 
$${T_{0}}_*: K_0(\m V/X) \to H_*(X)\otimes \bQ,$$
such that for a nonsingular variety $X$, ${T_{0}}_*([X \xrightarrow {\op{id}_X} X]) = td(TX) \cap [X].$
\item ($y=\, \, \, \, \,  1$): There exists a unique natural transformation 
$${T_{1}}_*: K_0(\m V/X) \to H_*(X)\otimes \bQ,$$
such that for a nonsingular variety $X$, ${T_{1}}_*([X \xrightarrow {\op{id}_X} X]) = L(TX) \cap [X].$
\end{enumerate}


Classically, similar natural transformations satisfying the same normalization conditions have been defined as follows:
\begin{enumerate}
\item MacPherson's Chern class \cite{MacPherson}: \emph{There exists a unique natural transformation 
$$c_*: F(X) \to H_*(X),$$ 
such that for a nonsingular variety $X$, $c_*(\jeden_X) = c(TX) \cap [X]$.} Here $F$ is the covariant functor assigning to $X$ the abelian group $F(X)$ of constructible functions on $X$. For a (possibly singular) variety $X$,   $$c_*(X):= c_*(\jeden_X)$$  is called the Chern--Schwartz--MacPherson class of $X$\footnote{Terminology is motivated by the fact that J.-P. Brasselet and M.-H. Schwartz \cite{BS} (see also \cite{AB})
showed that, for $X$ embedded in the complex manifold $M$, the MacPherson Chern class $c_*(\jeden_X)$ corresponds to the Schwartz class $c^S(X) \in H^{*}_X(M) = H^*(M, M \setminus X)$ (see \cite{Schw1, Schw2}) by  Alexander duality.}.
\item Baum--Fulton--MacPherson's Todd class \cite{BFM}: \emph{There exists a unique natural transformation 
$$td_*: G_0(X) \to H_*(X)\otimes \bQ,$$
such that for a nonsingular variety $X$, $td_*(\mathcal O_X) = td(TX) \cap [X]$.}
Here, $G_0$ is the covariant functor assigning to $X$ the Grothendieck group $G_0(X)$ of coherent sheaves on $X$, and $$td_*(X):=td_*(\mathcal O_X)$$ is called the Todd class of the (possibly singular) variety $X$.
\item Goresky-- MacPherson's homology $L$-class \cite{GM}, which is extended as a natural transformation by S. Cappell and J. Shaneson \cite{CS} (also see \cite{Yokura-TAMS}): \emph{For a compact variety $X$, there exists a unique natural transformation 
$$L_*: \Omega(X) \to  H_*(X)\otimes \bQ,$$
such that for a nonsingular variety $X$, $L_*(\mathbb Q_X[\op{dim}_{\bC}X]) = L(TX) \cap [X]$.}
Here, $\Omega$ is the covariant functor assigning to $X$ the cobordism group $\Omega(X)$ of self-dual constructible sheaf complexes on $X$. The value $$L_*(X):=L_*(IC_X)$$ on the Deligne intersection sheaf complex \cite{GM2} is called the homology $L$-class of $X$.
\end{enumerate}

The motivic Hirzebruch class transformation ${T_y}_*: K_0(\m V/X) \to H_*(X)\otimes \bQ[y]$ ``unifies" the above three characteristic classes $c_*, td_*, L_*$ in the sense that there exist commutative diagrams (see \cite{BSY1}):
$$\xymatrix{
& K_0(\Cal V/X)  \ar [dl]_{\epsilon} \ar [dr]^{{T_{-1}}_*} \\
{F(X) } \ar [rr] _{c_*\otimes \mathbb Q}& &  H_*(X)\otimes \bQ.}
$$
$$\xymatrix{
&  K_0(\Cal V/X)  \ar [dl]_{\Gamma} \ar [dr]^{{T_{0}}_*} \\
{G_0(X) } \ar [rr] _{td_*}& &  H_*(X)\otimes \bQ.}
$$
$$\xymatrix{
& K_0(\Cal V/X)  \ar [dl]_{\omega} \ar [dr]^{{T_{1}}_*} \\
{\Omega(X) } \ar [rr] _{L_*}& &  H_*(X)\otimes \bQ.}
$$
This ``unification" should be viewed as a positive answer to the following remark which is stated at the very end of MacPherson's survey article \cite{MacPherson2} (cf. \cite{Yokura-Banach}) 
\emph{``It remains to be seen whether there is a unified theory of characteristic classes of singular varieties like the classical one outlined above."}\footnote{At that time, the Goresky--MacPherson homology $L$-class was not yet available; it was defined only after the theory of intersection homology \cite{GM} was introduced by M. Goresky and R. MacPherson in 1980.} 
\begin{rem}\label{ide} It was shown in \cite{BSY1} that $T_{-1*}(X)=c_*(X)\otimes \mathbb Q$. However, in general, $T_{0*}(X) \not =td_*(X)$ and $T_{1*}(X) \not= L_*(X)$. Furthermore, it was shown in \cite{BSY1} that $T_{0*}(X) =td_*(X)$ if $X$ has at most Du Bois singularities (e.g., rational singularities), and it was 
conjectured in \cite{BSY1,SY} that $T_{1*}(X) = L_*(X)$ if $X$ is a rational homology manifold. For instances where this conjecture has been proven, see \cite{Banagl, CMSS0, CMSS, MS2}. In particular, if $X$ is a toric variety, then $T_{0*}(X) =td_*(X)$, and it was shown in \cite{MS2} that if $X$ is a simplicial projective toric variety, then  
$T_{1*}(X) = L_*(X)$. \end{rem}

\begin{defn}\label{Hc}
Following \cite{BSY1}, we call 
\begin{center} $td_*^H(X):=T_{0*}(X)$ and $L_*^H(X):=T_{1*}(X)$ \end{center}
 the \emph{Hodge--Todd class} and,  respectively, the \emph{Hodge $L$-class} of $X$.
Similarly,  \begin{center} $\tau^H(X):=\chi_{0}(X)$ and $\sigma^H(X):=\chi_1(X)$ \end{center}
will be called the \emph{Hodge--Todd genus} and, respectively, the \emph{Hodge signature} of $X$.
\end{defn}

\section{Congruence formulae for motivic Hirzebruch classes}

Recall that the Hirzebruch $\chi_y$-genus $\chi_y(X)$ of a (possibly singular) complex $n$-dimensional algebraic variety $X$, is a degree $n$ polynomial expressed as:
$$\chi_y(X) = \chi^0(X)+ \chi^1(X)y + \chi^2(X)y^2+\cdots +\chi^n(X)y^n  \, \, \, \in \mathbb Z[y].$$
Similarly, as shown in \cite{BSY1,Sch}, the motivic Hirzebruch class $T_{y*}(X)$ can also be expressed as:
$$T_{y*}(X) =T^0_*(X) + T^1_*(X)y + T^2_*(X)y^2 + \cdots + T^n_*(X)y^n \, \, \, \in H_*(X)\otimes \mathbb Q[y].$$

Let us now assume that $X$ is a \emph{compact} variety. Then the $\chi_y$-genus $\chi_y(X)\in \bZ[y]\subset \bQ[y]$ is the degree of the motivic Hirzebruch class $T_{y*}(X)$.

By Definition \ref{Hc} and Remark \ref{ide} we have that
\begin{enumerate}
\item $c_*(X)\otimes \mathbb Q = T_{-1*}(X)=  T^0_*(X) - T^1_*(X) + T^2_*(X) - T^3_*(X)  \cdots + (-1)^nT^n_*(X).$
\item $td_*^H(X)  = T_{0*}(X)=  T^0_*(X)$
\item $L_*^H(X) = T_{1*}(X)=  T^0_*(X) +T^1_*(X) + T^2_*(X) + T^3_*(X)  \cdots + T^n_*(X).$
\end{enumerate}
and, similarly,
\begin{enumerate}
\item $\chi(X) =\chi^0(X)- \chi^1(X) + \chi^2(X)+\cdots +\chi^n(X)(-1)^n$
\item $\tau^H(X)  = \chi_0(X) = \chi^0(X).$
\item $\sigma^H(X) = \chi_1(X) = \chi^0(X)+ \chi^1(X) + \chi^2(X)+\cdots +\chi^n(X).$
\end{enumerate}
For convenience, let us introduce the following notations:
\begin{itemize}
\item $T_*^{\op{even}}(X) := \sum_{i\geqq 0} T_*^{2i}(X) \,\, \text{the even part}, \quad T_*^{\op{odd}}(X) := \sum_{i\geqq 0} T_*^{2i+1}(X) \, \, \text{the odd part},$
\item $\chi^{\op{even}}(X) := \sum_{i\geqq 0} \chi^{2i}(X) \,\, \text{the even part}, \quad \chi^{\op{odd}}(X) := \sum_{i\geqq 0} \chi^{2i+1}(X) \, \, \text{the odd part}.$
\end{itemize}
Then we have the following identities:
\begin{equation}\label{formula1}
L_*^H(X) + c_*(X)\otimes \mathbb Q  = 2 T_*^{\op{even}}(X), \quad L_*^H(X) - c_*(X)\otimes \mathbb Q  = 2 T_*^{\op{odd}}(X),
\end{equation}
\begin{equation}\label{formula1b}
\sigma^H(X) + \chi(X)= 2 \chi^{\op{even}}(X), \quad \sigma^H(X) - \chi(X) = 2 \chi^{\op{odd}}(X),
\end{equation}
from which we get the following congruence formula:
\begin{lem}\label{lem1} With the above notations and assumptions, we have:
$$T_{y*}(X) \equiv \frac{c_*(X)\otimes \mathbb Q}{2} (1-y) +\frac{L_*^H(X)}{2} (1+y)  \,\, \, \, \op{mod} \Bigl (H_{*}(X) \otimes \mathbb Q[y] \Bigr)(1-y^2).$$
\end{lem}
\begin{proof} 
If we let $y^2=1$ in $T_{y*}(X) =T^0_*(X) + T^1_*(X)y + T^2_*(X)y^2 + \cdots + T^n_*(X)y^n$, then we have
\begin{align*}
& T_{y*}(X)  = \sum T_*^i(X)y^i \\
& \equiv T_*^0(X) + T_*^1(X)y + T_*^2(X) +T_*^3(X)y +T_*^4(X) + \cdots   \, \, \, \op{mod} \Bigl (H_{*}(X) \otimes \mathbb Q[y] \Bigr)(1-y^2)\\
& = T_*^{\op{even}}(X) + T_*^{\op{odd}}(X)y \, \, \,\, \,\, \,\op{mod} \Bigl (H_{*}(X) \otimes \mathbb Q[y] \Bigr)(1-y^2)\\
 & \overset{(\ref{formula1})}{=} \frac{L_*^H(X) + c_*(X)\otimes \mathbb Q}{2} + \frac{L_*^H(X) - c_*(X)\otimes \mathbb Q}{2}y   \, \, \,\,  \op{mod} \Bigl (H_{*}(X) \otimes \mathbb Q[y] \Bigr)(1-y^2)  \\
& =\frac{c_*(X)\otimes \mathbb Q}{2}(1-y) +\frac{L_*^H(X)}{2}(1+y) \,\,   \,\, \op{mod} \Bigl (H_{*}(X) \otimes \mathbb Q[y] \Bigr)(1-y^2) .
\end{align*}
\end{proof}
\begin{rem}\label{rem} Since $X$ is compact, by taking the degree in the formula of Lemma \ref{lem1} we get:
$$\chi_y(X)  \equiv \frac{\chi(X)}{2} (1-y) +\frac{\sigma^H(X)}{2} (1+y)  \,\, \, \, \op{mod} \bigl (\mathbb Z[y] \bigr)(1-y^2) \, \,  \text {i.e.},\op{mod} \, 1-y^2.$$
\end{rem}
Finer congruences involving also 
the Hodge--Todd class $td_*^H(X) =T_{0*}(X)=T^0_*(X)$ can be obtained as follows:
\begin{lem}\label{lem2}
\begin{align*}
& T_{y*}(X)  \equiv \frac{c_*(X)\otimes \mathbb Q}{2}(y^2 -y) + td_*^H(X)(1-y^2) + \frac{L_*^H(X)}{2}(y+y^2) \\
& \hspace{8cm} \op{mod} \Bigl (H_{*}(X) \otimes \mathbb Q[y] \Bigr)(y-y^3).
\end{align*}
\end{lem}
\begin{proof}
If we let $y^3=y$ 
in $T_{y*}(X) =T^0_*(X) + T^1_*(X)y + T^2_*(X)y^2 + \cdots + T^n_*(X)y^n$, then we have
\begin{align*}
& T_{y*}(X)  = \sum T_*^i(X)y^i \\
& \equiv T_*^0(X) + T_*^1(X)y + T_*^2(X)y^2 +T_*^3(X)y +T_*^4(X)y^2 + \cdots \\
& \hspace{9cm} 
\,\op{mod} \Bigl (H_{*}(X) \otimes \mathbb Q[y] \Bigr)(y-y^3)
\\
& = T_*^0(X) + T_*^{\op{odd}}(X)y + \Bigl (T_*^{\op{even}}(X) - T_*^0(X) \Bigr )y^2 \, \,\, \op{mod} \Bigl (H_{*}(X) \otimes \mathbb Q[y] \Bigr)(y-y^3) \\
 & \overset{(\ref{formula1})}{=}  td_*^H(X) + \frac{L_*^H(X) - c_*(X)\otimes \mathbb Q}{2}y + \left(\frac{L_*^H(X) + c_*(X)\otimes \mathbb Q}{2} -td_*^H(X) \right) y^2  \,\\
 & \hspace{8cm}  \op{mod} \Bigl (H_{*}(X) \otimes \mathbb Q[y] \Bigr)(y-y^3)  \\
& =  \frac{c_*(X)\otimes \mathbb Q}{2}(y^2 -y) + td_*^H(X)(1-y^2) + \frac{L_*^H(X)}{2}(y+y^2) \\
& \hspace{8cm}\op{mod} \Bigl (H_{*}(X) \otimes \mathbb Q[y] \Bigr)(y-y^3).
\end{align*}
\end{proof}
\begin{rem} 
Since $X$ is compact, by taking the degree in the formula of Lemma \ref{lem2} we get:
$$\chi_y(X) \equiv \frac{\chi(X)}{2}(y^2 -y) + \tau^H(X)(1-y^2) + \frac{\sigma^H(X)}{2}(y+y^2) \, \, \op{mod} \,\, y-y^3. \hspace{2cm}$$
\end{rem}
\begin{rem} We note that if we let $y^2=1$ in the congruence formula of Lemma \ref{lem2}, we get the formula from Lemma \ref{lem1}, so from this point of view Lemma \ref{lem2} presents a finer congruence identity.
\end{rem} 
Combining the results of Lemma \ref{lem1} and Lemma \ref{lem2},  we get the following congruence formulae:

\begin{cor}\label{corollary} Let $f:X \to Z$ and $g:Y \to Z$ be morphisms of complex algebraic varieties with the same target variety $Z$, and assume that $X$, $Y$ and $Z$  are compact. Then we have the following congruence identities for the difference $f_*T_{y*}(X) - g_*T_{y*}(Y)$:
\begin{align*}
(1) \, & f_*T_{y*}(X) - g_*T_{y*}(Y) \equiv \frac{f_*c_*(X)\otimes \mathbb Q -g_*c_*(Y)\otimes \mathbb Q}{2}(1-y)+ \frac{f_*L_*^H(X) -g_*L_*^H(Y)}{2}(1+y) 
 \\
& \hspace{9cm} \op{mod} \Bigl (H_{*}(Z) \otimes \mathbb Q[y] \Bigr)(1-y^2).
\end{align*}
\begin{align*}
(2) \, & f_*T_{y*}(X) - g_*T_{y*}(Y) \equiv \frac{f_*c_*(X)\otimes \mathbb Q -g_*c_*(Y)\otimes \mathbb Q}{2}(y^2 -y) 
\\
& \hspace{3cm} +\Bigl( f_*td_*^H(X)-g_*td_*^H(Y) \Bigr)(1-y^2) 
+ \frac{f_*L_*^H(X) -g_*L_*^H(Y)}{2}(y+y^2) \\
& \hspace{9cm} \op{mod} \Bigl (H_{*}(Z) \otimes \mathbb Q[y] \Bigr)(y-y^3).
\end{align*}
In particular, by taking the degrees in (1) and (2), we have:
$$ (3) \, \, \chi_y(X) - \chi_y(Y) \equiv  \frac{\chi(X) -\chi(Y)}{2}(1-y) + \frac{\sigma^H(X) -\sigma^H(Y)}{2}(1+y)  \, \, \,\, \op{mod} \, 1-y^2. \qquad \qquad$$
\begin{align*}
(4) \, \, \, \chi_y(X) - \chi_y(Y) \equiv & \frac{\chi(X) -\chi(Y)}{2}(y^2 -y) + \left (\tau^H(X)-\tau^H(Y) \right )(1-y^2) \hspace{3cm} \\
&   \hspace{4cm}  +  \frac{\sigma^H(X) -\sigma^H(Y)}{2}(y+y^2)  \,\, \, \,\,\op{mod} \, y-y^3
\end{align*}

\end{cor}
\begin{rem} The above formulae (3) and (4) can also be regarded as special cases of (1) and (2), respectively, for the case when the target variety $Z$ is a point.

\end{rem}
\begin{rem} In category theory, a pair of morphisms $f:X \to Z$ and $g:Y \to Z$ with the same target $Z$ 
is called \emph{a cospan $X \xrightarrow f Z \xleftarrow {g} Y$}. This notion is dual to the notion of \emph{span} (or \emph{correspondence}), which consists of a pair of  morphisms $f:Z \to X$ and $g:Z \to Y$ with the same source variety: $X \xleftarrow f Z \xrightarrow g Y$.  The formulae of Corollary \ref{corollary} can be regarded as congruence formulae for cospans of complex algebraic varieties.
\end{rem}

\section{Application to complex algebraic fiber bundles}
In this section, we consider a cospan of the form: 
$$E \xrightarrow f B \xleftarrow {g} F \times B$$
where $E, F, B$ are complex algebraic varieties, and $g=pr_2:F \times B \to B$ is the second factor projection with fiber the compact algebraic variety $F$.

Recall from \cite{BSY1} that the motivic Hirzebruch class transformation $T_{y*}$ commutes with cross-products (cf. \cite{K, KY}). In particular, we have:
\begin{pro}[Cross product formula] 
\begin{equation}\label{cp} T_{y*}(X \times Y) = T_{y*}(X) \times T_{y*}(Y).\end{equation}
\end{pro}
\begin{cor} In the above notations, we have: 
\begin{equation}\label{proj} (pr_2)_*T_{y*}(F \times B)= \chi_y(F)T_{y*}(B).\end{equation}
\end{cor}
\begin{proof} Let $a_F:F \to pt$ be the map to a point. Then the projection $pr_2:F \times B \to B$ is the same as the product of two morphisms $\op{id}_B:B \to B$ and $a_F:F \to pt$:
$$pr_2:=a_F \times \op{id}_B:F \times B \to pt \times B = B.$$
Hence we have
\begin{equation*}\begin{split}
(pr_2)_*T_{y*}(F \times B) &  \;  \;   = (a_F \times \op{id}_B)_*T_{y*}(F \times B) \\
& \overset{(\ref{cp})}{=} (a_F \times \op{id}_B)_*\left(T_{y*}(F)  \times T_{y*}(B) \right)\\
&  \;  \;  = (a_F)_* \times (\op{id}_B)_* \left(T_{y*}(F) \times T_{y*}(B) \right)\\
&  \;  \;  =  (a_F)_*T_{y*}(F)  \times (\op{id}_B)_*T_{y*}(B)\\
& \overset{(\ref{deg})}{=}  \chi_y(F) \times T_{y*}(B) \\
&  \;  \;  = \chi_y(F)T_{y*}(B).
\end{split}
\end{equation*}
\end{proof}

\begin{thm}\label{thm5} Let $f:E \to B$ be a proper morphism of complex algebraic varieties. Assume that all fibers of $f$ have the same Euler characteristic $\chi(F)$ as the compact algebraic variety $F$. Then we have: 
\begin{align*}
 (1) \, \, f_*T_{y*}(E) -& \chi_y(F)T_{y*}(B) \\
& \equiv \frac{f_*L_*^H(E) -\sigma^H(F)L_*^H(B)}{2} (1+y)  \,\, \,\, \,\,\op{mod} \Bigl (H_{*}(B) \otimes \mathbb Q[y] \Bigr)(1-y^2).
\end{align*}
If $B$ (and therefore $E$) is compact, then by taking the degree in (1) we have:
$$\chi_y(E) -\chi_y(F)\chi_y(B) \equiv \frac{\sigma^H(E)-\sigma^H(F)\sigma^H(B)}{2}(1+y) \,\, \op{mod} \, \, 1-y^2.$$
\begin{align*}
(2) \, \, f_*T_{y*}(E) -\chi_y(F)T_{y*}(B) & \equiv \Bigl (f_*td_*^H(E) -\tau^H(F)td_*^H(B) \Bigr)(1-y^2) \\
& \hspace{1cm}+ \frac{f_*L_*^H(E)-\sigma^H(F)L_*^H(B)}{2}(y+ y^2) \\
& \hspace{4cm}\op{mod} \Bigl (H_{*}(B) \otimes \mathbb Q[y] \Bigr)(y-y^3).
\end{align*}
If $B$ (and therefore $E$) is compact, then by taking the degree in (2) we have:
\begin{align*}
\chi_y(E) -\chi_y(F)\chi_y(B) & \equiv \Bigl (\tau^H(E) -\tau^H(F)\tau^H(B) \Bigr)(1-y^2)\\
& \hspace{2cm} + \frac{\sigma^H(E)-\sigma^H(F)\sigma^H(B)}{2}(y^2+y) \,\, \op{mod} \, \, y-y^3.
\end{align*}
\end{thm}
\begin{proof} By our assumption on the fibers of $f$, we have that $f_*\jeden_E=\chi(F)\cdot \jeden_B$. Consider the projection $pr_2:F \times B \to B$. Then it follows from Corollary \ref{corollary} that
\begin{align*}
& f_*T_{y*}(E) - (pr_2)_*T_{y*}(F\times B) \equiv \frac{f_*c_*(E)\otimes \mathbb Q -(pr_2)_*c_*(F \times B)\otimes \mathbb Q}{2}(1-y) \\
& \hspace{7cm} + \frac{f_*L_*^H(E) -(pr_2)_*L_*(F\times B)}{2}(1+y) \\
& \hspace{8cm} \op{mod} \Bigl (H_{*}(B) \otimes \mathbb Q[y] \Bigr)(1-y^2).
\end{align*}
By (\ref{proj}), we have that $(pr_2)_*T_{y*}(F\times B) = \chi_y(F)T_{y*}(B)$, so in particular, $(pr_2)_*c_*(F \times B)=\chi(F)c_*(B)$ and $(pr_2)_*L^H_*(F\times B) = \sigma^H(F)L_*^H(B).$
Furthermore, it is well-known that $f_*c_*(E) = \chi(F)c_*(B)$. In fact, 
\begin{align*}
f_*c_*(E) & = f_*c_*(\jeden_E) \\
& = c_*(f_*\jeden_E) \\
& = c_*(\chi(F)\jeden_B) \\
& = \chi(F) c_*(\jeden_B) \\
& = \chi(F)c_*(B).
\end{align*}
Therefore we get the congruence (1). The congruence (2) is obtained in a similar manner.
\end{proof}
\begin{rem} An alternative  proof of Theorem \ref{thm5}, without using the cross-product formula can be given as follows:

(1) Since  

$\displaystyle \chi_y(F) \equiv \frac{\chi(F)}{2} (1-y)  + \frac{\sigma^H(F)}{2} (1+y) \,\, \op{mod} \, 1-y^2,$

$\displaystyle T_{y*}(B) \equiv \frac{c_*(B)\otimes \mathbb Q}{2} (1-y) + \frac{L_*^H(B)}{2}(1+y) \hspace{0.5cm}  \op{mod} \Bigl (H_{*}(B) \otimes \mathbb Q[y] \Bigr)(1-y^2),$

\noindent we have, $\op{mod} \Bigl (H_{*}(B) \otimes \mathbb Q[y] \Bigr)(1-y^2)$, that  
\begin{align*}
\chi_y(F)T_{y*}(B) & \equiv \left(\frac{\chi(F)}{2} (1-y)  +  \frac{\sigma^H(F)}{2} (1+y) \right) \left( \frac{c_*(B)\otimes \mathbb Q}{2} (1-y) +\frac{L_*^H(B)}{2} (1+y) \right)\\
& = \frac{\chi(F)c_*(B)\otimes \mathbb Q}{4}(1-y)^2 + \frac{\chi(F)L_*^H(B)}{4}(1-y^2) \\
& \hspace{3.5cm} + \frac{\sigma^H(F)c_*(B)\otimes \mathbb Q}{4}(1-y^2)  + \frac{\sigma^H(F)L_*^H(B)}{4}(1+y)^2. 
\end{align*}
Since $(1-y)^2 = 1 - 2y + y^2 \equiv 2(1-y) \,\op{mod} \, 1-y^2$ and $(1+y)^2 = 1 + 2y + y^2 \equiv 2(1+y) \, \op{mod} \, 1-y^2$, the above congruence becomes the following: $\op{mod} \Bigl (H_{*}(B) \otimes \mathbb Q[y] \Bigr)(1-y^2)$, 
$$\chi_y(F)T_{y*}(B) \equiv  \frac{\chi(F)c_*(B)\otimes \mathbb Q}{2}(1-y) +\frac{\sigma^H(F)L_*^H(B)}{2}(1+y).$$
Therefore we have
\begin{align*}
& f_*T_{y*}(E) - \chi_y(F)T_{y*}(B) \equiv  \frac{f_*c_*(E)\otimes \mathbb Q- \chi(F)c_*(B)\otimes \mathbb Q}{2}(1-y) \\
&  \hspace{3cm} +\frac{f_*L_*^H(E) -\sigma^H(F)L_*^H(B)}{2}(1+y) 
\,\, \,  \op{mod} \Bigl (H_{*}(B) \otimes \mathbb Q[y] \Bigr)(1-y^2).
\end{align*}
Hence, by using $f_*c_*(E) = \chi(F)c_*(B)$ we get the congruence formula (1) of Theorem \ref{thm5}.

(2)  Similarly, by Lemma \ref{lem2}, we can compute $\chi_y(F)T_{y*}(B)$ $\op{mod} \Bigl (H_{*}(B) \otimes \mathbb Q[y] \Bigr)(y-y^3)$ as follows:
\begin{align*}
& \chi_y(F)T_{y*}(B) \equiv \\
& \left(\frac{\chi(F)}{2}(y^2 -y) + \tau^H(F)(1-y^2) + \frac{\sigma^H(F)}{2}(y+y^2)\right) \\
& \hspace{2cm} \times \left( \frac{c_*(B)\otimes \mathbb Q}{2}(y^2 -y) + td_*^H(B)(1-y^2) + \frac{L_*^H(B)}{2}(y+y^2) \right) \\
& \equiv \frac{\chi(F)c_*(B)\otimes \mathbb Q}{4}(y^2 -y)^2 + \frac{\chi(F)}{2}td_*^H(B)(y^2-y)(1-y^2)   + \frac{\chi(F)L_*^H(B)}{4}(y^2-y)(y+y^2) \\
& + \frac{\tau^H(F)c_*(B)\otimes \mathbb Q}{2}(1-y^2)(y^2 -y)  +  \tau^H(F)td_*^H(B)(1-y^2)^2  + \frac{\tau^H(F)L_*^H(B)}{2}(1-y^2)(y+y^2) \\
&  + \frac{\sigma^H(F) c_*(B)\otimes \mathbb Q}{4}(y+y^2)(y^2 -y) +\frac{\sigma^H(F)}{2}td_*^H(B)(y+y^2)(1-y^2) + \frac{\sigma^H(F)L_*^H(B)}{4}(y+y^2)^2.
\end{align*}
Since $(1-y^2)(y^2-y)=(1-y^2)y(y-1)\equiv 0 \op{mod} y-y^3$, $(y^2-y)(y+y^2)=y(y-1)y(1+y)=y^2(1-y^2) \equiv 0 \op{mod} y-y^3$ and $(1-y^2)(y+y^2)=(1-y^2)y(1+y)\equiv 0 \op{mod} y-y^3$, we obtain the following congruence:
\begin{align*}
\chi_y(F)T_{y*}(B) \equiv \frac{\chi(F)c_*(B)\otimes \mathbb Q}{4}(y^2 -y)^2 + \tau^H(F)td_*^H(B)(1-y^2)^2 + \frac{\sigma^H(F)L_*^H(B)}{4}(y+y^2)^2. \hspace{2cm}
\end{align*}
Since $(y^2-y)^2 \equiv 2(y^2-y) \op{mod}\, y-y^3$ , $(1-y^2)^2 \equiv 1-y^2 \op{mod}\, y-y^3$ and $(y^2+y)^2 \equiv 2(y^2+y) \op{mod}\,y-y^3$, 
the above congruence becomes
$$\chi_y(F)T_{y*}(B) \equiv \frac{\chi(F)c_*(B)\otimes \mathbb Q}{2}(y^2 -y)  + \tau^H(F)td_*^H(B)(1-y^2) + \frac{\sigma^H(F)L_*^H(B)}{2}(y+y^2).$$
Hence, we have
\begin{align*}
& f_*T_{y*}(E) - \chi_y(F)T_{y*}(B) \equiv \frac{f_*c_*(E)\otimes \mathbb Q - \chi(F)c_*(B)\otimes \mathbb Q }{2}(y^2 -y)
 \\
& \hspace{3cm} + \left(f_*td_*^H(E) - \tau^H(F)td_*^H(B) \right) (1-y^2) + \frac{f_*L_*^H(E) -\sigma^H(F)L_*^H(B)}{2}(y+y^2). 
\end{align*}
Since $f_*c_*(E)=\chi(F)c_*(B)$, we get
\begin{align*}
f_*T_{y*}(E) -\chi_y(F)T_{y*}(B) & \equiv \Bigl (f_*td_*^H(E) -\tau^H(F)td_*^H(B) \Bigr)(1-y^2) \\
& \hspace{2cm}+ \frac{f_*L_*^H(E)-\sigma^H(F)L_*^H(B)}{2}(y+ y^2) \\
& \hspace{4cm}\op{mod} \Bigl (H_{*}(B) \otimes \mathbb Q[y] \Bigr)(y-y^3).
\end{align*}
\end{rem}

In the rational homology class results of Theorem  \ref{thm5} one can only specialize the parameter $y$ to $y=\pm 1$ in (1) resp. $y=0,\pm 1$ in (2),
since these are the only zeros of $1-y^2$, resp., $y-y^3$ over $\bQ$. But the corresponding degree formulae live in $\bZ[y]/(1-y^2)$, resp., $\bZ[y]/(y-y^3)$, in which case other specializations are also possible, e.g., $y^2-1=4m(m+1)$ is divisible by $8$ for $y=2m+1$ an odd integer, or $y^3-y=(y-1)y(y+1)$ is divisible by $6$ for any integer $y$.

Thus, if we consider concrete integers $y$, and look only at the degree formulae of Theorem \ref{thm5}, we have the following:
\begin{cor} Under the notations of Theorem \ref{thm5}, we have:
\begin{enumerate}
\item For any odd integer $y$ we have 
$$\chi_y(E) \equiv \chi_y(F)\chi_y(B) \, \, \op{mod} 2.$$
\item If $y \equiv 3 \, \,\op{mod} 4$, then 
$$\chi_y(E) \equiv \chi_y(F)\chi_y(B) \, \, \op{mod} 4.$$
\item If $y \equiv 1 \, \,\op{mod} 4$, then 
$$\chi_y(E) \equiv \chi_y(F)\chi_y(B) \, \, \op{mod} 4 \Longleftrightarrow \sigma^H(E) \equiv \sigma^H(F)\sigma^H(B) \, \, \op{mod} \, \, 4.$$
\end{enumerate}
\end{cor}

\begin{proof} It suffices to observe that $\sigma^H(E) -\sigma^H(F)\sigma^H(B)$ is an even number, i.e., 
$$\sigma^H(E) \equiv \sigma^H(F)\sigma^H(B) \, \, \op{mod} \, 2.$$
This follows from (\ref{formula1b}). Indeed,
$\sigma^H(E) + \chi(E) \equiv 0 \, \op{mod} \, 2$ and $\sigma^H(F \times B) + \chi(F \times B) \equiv 0 \, \op{mod} \, 2$ imply that
$\sigma^H(E) + \chi(E) - (\sigma^H(F \times B) + \chi(F \times B)) \equiv 0 \, \op{mod} \, 2$. Since  $\sigma^H(F \times B) = \sigma^H(F) \sigma^H(B)$ and $\chi(E) = \chi(F)\chi(B) = \chi(F \times B)$, we get $\sigma^H(E) - \sigma^H(F) \sigma^H(B) \equiv 0 \, \op{mod} \, 2$.  Thus we have $\sigma^H(E) \equiv \sigma^H(F)\sigma^H(B) \, \, \op{mod} \, 2.$ So, if $y$ is an odd integer, then $1+y$ is an even number and $1-y^2$ is also an even numer (in fact, it is divisible by $8$). If $y \equiv 3 \, \,\op{mod} 4$, i.e., $y=4k+3$, then $1+y =4k+4$ is divisible by $4$ and $1-y^2$ is also divisible by $4$ (since it is divisible by $8$). In a similar way we get (3). Thus we get the above results.

\end{proof} 
\begin{rem} If $F, E, B$ are smooth, then in the above Corollary one can replace $\op{mod} 2$ in (1) and $\op{mod} 4$ in (2) and (3)  by $\op{mod} 4$  and $\op{mod} 8$, respectively, as recalled in Theorem \ref{mod8} of the Introduction. The reason is the use of the duality result of Theorem \ref{duality},
which for the $\chi_y$-genus of singular varieties is not available. But, as we will see in \S6 (Theorem \ref{ihduality}), a corresponding duality result can be formulated in the singular setting by using intersection (co)homology.

\end{rem}

\begin{rem} If $y$ is an even integer, we can use the degree congruence in Theorem \ref{thm5} (2) to get the following: if $y=2m$, then
$$\chi_{2m}(E) \equiv \chi_{2m}(F)\chi_{2m}(B) \, \op{mod} \, 2m \Longleftrightarrow \tau^H(E) \equiv \tau^H(F)\tau^H(B) \, \op{mod} \, 2m.$$
\end{rem}

We also mention here the following consequence of the degree formula in Theorem \ref{thm5} (2):
\begin{cor}
 Under the assumption of Theorem \ref{thm5}, we have for $B$ (and therefore also $E$) compact, and $y \equiv\; 2$ or $3 \op{mod} 6$: 
 $$\chi_y(E) - \chi_y(F)\chi_{y}(B) \equiv \;\left( \tau^H(E)-\tau^H(F)\tau^H(B) \right) (1-y^2) \op{mod} 6  \:.$$
\end{cor}

In general, the Hodge--Todd genus $\tau^H(-)$ is not multiplicative for such maps $f: E\to B$, so that the right had side in the above congruence identity does not necessarily vanish. But in the following example, this is indeed the case. 

\begin{ex} Assume that $f: E\to B$ is a proper submersion of algebraic varieties,
with the compact algebraic rational manifold $F$ isomorphic to a fiber of $f$ over every connected component of $B$. 
Since $f$ is a proper submersion, all coherent higher image sheaves 
$R^if_*\mathcal{O}_E$ are locally free and commute with base change (compare, e.g., with   \cite[Section 3.A]{Sch}). But $H^0(F,\mathcal{O})=\bC$ and
$\dim H^i(F,\mathcal{O})=\dim H^0(F, \Omega^i) =0$ for $i>0$, since $F$ is rational,
i.e., connected and birational to some projective space. So $f_*\mathcal{O}_E\simeq \mathcal{O}_B$ and $R^if_*\mathcal{O}_E = 0$  for $i>0$.
By \cite[Corollary 5.12]{Sch}, this implies that $f_*td_*^H(E)=td^H_*(B)$, and $\tau_*^H(E)=\tau^H_*(B)$ together with $\tau^H(F)=1$.
\end{ex}

\section{Intersection homology flavored congruence identities}

As shown in \cite{BSY1,Sch}, the motivic Hirzebruch class transformation $T_{y*}:K_0(\mathcal V /X) \to \bQ[y]\subset \bQ[y^{\pm 1}]$ factors through the Grothendieck group $K_0(\mh(X))$ of (algebraic) mixed Hodge modules on $X$ \cite{Saito}, by a transformation 
$$ MHT_{y*}:K_0(\mh(X)) \to H_{*}(X) \otimes \bQ[y^{\pm 1}],$$
together with the natural group homomorphism
$$
\chi_{\rm Hdg}: K_0(\mathcal V /X) \to K_0(\mh(X)) \ , \ [f:Y \to X] \mapsto [f_! \bQ_Y] \ ,
$$
where $\bQ_Y$ is considered here as the constant mixed Hodge module complex on $Y$. So the motivic Hirzebruch class ${T_y}_*(X)$ can also be defined as $MHT_{y*}([\bQ_X])$. 

One advantage of using mixed Hodge modules is that we can evaluate the transformation $MHT_{y*}$ on other interesting ``coefficients''. For example, the (shifted) intersection cohomology Hodge module $IC'_X:=IC_X[-\dim (X)]$ yields, for $X$ pure-dimensional, similar \emph{intersection Hirzebruch classes} $$IT_{y*}(X):=MHT_{y*}([IC'_X]) \in H_{*}(X) \otimes \bQ[y].$$ 
The fact that no negative powers of $y$ appear in $IT_{y*}(X)$ follows from \cite[Example 5.2]{Sch}, where it is also shown that the highest power of $y$ appearing in $IT_{y*}(X)$ is at most $\dim_{\bC}(X)$. In this section, we discuss congruence formulae for the classes ${IT_y}_*(X)$, as well as congruences for the corresponding degree generalizing Theorem \ref{mod8}.

\medskip

The actual definition of $MHT_{y*}$ is not needed here. We list only those formal properties needed for obtaining congruence identities.

The transformation $MHT_{y*}$ is functorial for proper morphisms. In particular, if $X$ is compact and pure-dimensional, then the degree-zero component of $IT_{y*}(X)$ is yet another extension of the Hirzebruch $\chi_y$-genus to the singular context, namely the \emph{intersection Hodge polynomial} $I\chi_y(X)$ defined by making use of Saito's (mixed) Hodge structures on the intersection cohomology groups $IH^i(X,\bQ):=\mathbb{H}^i(X,IC'_X)$:
$$I\chi_y(X):=\sum_{i, p \geq 0} (-1)^i \op{dim}_{\mathbb C} Gr^p_{F} \left( IH^i (X, \mathbb C) \right) (-y)^p \in \bZ[y].$$
Note that since the highest power of $y$ appearing in $IT_{y*}(X)$ is $\dim_{\bC}(X)$, it  follows that $I\chi_y(X)$ is a polynomial of degree at most $\dim_{\bC}(X)$. Set 
$$I\chi_y(X)=:\sum_{p=0}^{\dim_{\bC}(X)} I\chi^p(X)y^p.$$
With this notation we have the following duality result generalizing Theorem \ref{duality} to the singular context:
\begin{thm}\label{ihduality}
Assume $X$ is a compact complex algebraic variety of pure dimension $n$. Then
$$I\chi_y(X)=(-y)^n\cdot I\chi_{\frac{1}{y}}(X), \quad i.e., \quad I\chi^p(X)=(-1)^n\cdot I\chi^{n-p}(X)\:.$$
\end{thm}

\begin{proof}
Let $\mathbb{D}$ denote the duality functor on mixed Hodge module complexes. Then
$\mathbb{D}(IC'_X)=IC'_X[2n](n)$, so $deg(MHT_{y*}(\mathbb{D}(IC'_X))) = I\chi_y(X)\cdot (-y)^{-n}$. On the other hand, by the duality result of \cite[Corollary 5.19]{Sch}, one also gets that $deg(MHT_{y*}(\mathbb{D}(IC'_X))) = I\chi_{\frac{1}{y}}(X)$.
\end{proof}

\begin{rem}\label{rhm}
Note that if $X$ is a rational homology manifold (e.g., $X$ is smooth or only has quotient singularities), then $IC'_X=\bQ_X$ as (shifted) mixed Hodge modules.  
So in this case one gets that $I\chi_y(X)=\chi_y(X)$ for $X$ compact. In particular, it follows from the above result that Theorem \ref{duality} holds for a compact complex algebraic variety of pure dimension $n$ which, moreover, is a rational homology manifold.
\end{rem}

The transformation $MHT_{y*}$ also commutes with external products, so the cross-product formula (\ref{cp}) also holds for $IT_{y*}$, namely,
$$IT_{y*}(X \times Y) = IT_{y*}(X) \times IT_{y*}(Y).$$
Moreover, the intersection homology variants of Theorems \ref{CMS} and \ref{CMS2} from the Introduction were proven in \cite[Corollary 4.11]{CMS}.

By \cite[Proposition 5.21]{Sch}, we have that  
$$IT_{-1*}(X)=c_*(ic_X)=:Ic_*(X) \in H_*(X) \otimes \bQ,$$
for $ic_X$ the constructible function on $X$ defined by taking stalkwise the Euler characteristic of the constructible sheaf complex $IC'_X$. In particular, at the degree level,
$I\chi_{-1}(X)=I\chi(X)$ is the intersection cohomology Euler characteristic.
Moreover, if $f:E \to B$ is a complex algebraic fiber bundle with fiber variety $F$, then we have by \cite[Proposition 3.6]{CMS0} that
$$f_* Ic_*(E)=I\chi(F) \cdot Ic_*(B).$$

Furthermore, it is conjectured in \cite[Remark 5.4]{BSY1} that $IT_{1*}(X)=L_*(X)$ is the homology $L$-class of Goresky--MacPherson. In the case when $X$ is pure-dimensional and projective, 
the identification of these homology classes holds at degree level by Saito's mixed Hodge module theory, i.e., one has in this case that \begin{equation}\label{ichisig} I\chi_1(X)=\sigma(X)\end{equation} is the Goresky--MacPherson intersection cohomology signature (see \cite[Section 3.6]{MSS}). 

\begin{rem}\label{odd}
Note that if $\dim_{\bC}(X)$ is odd, then $\sigma(X)=0$ by definition (since the pairing on middle dimensional interesction cohomology is anti-symmetric), whereas in this case the fact that $I\chi_1(X)=0$ follows from the duality result of Theorem \ref{ihduality}.
\end{rem}

\medskip

We therefore have all the ingredients to formally extend \emph{all} the results of the previous sections to intersection Hirzebruch classes. Let us only formulate here the following result needed in the proof of Theorem \ref{ihmod8} from the introduction:
\begin{thm}\label{thm6} Let $f:E \to B$ be a complex algebraic fiber bundle with fiber $F$, so that $E, B, F$ are pure-dimensional. Then we have: 
\begin{align*}
 (1) \, \, f_*IT_{y*}(E) -& I\chi_y(F)IT_{y*}(B) \\
& \equiv \frac{f_*IT_{1*}(E) -\sigma(F)IT_{1*}(B)}{2} (1+y)  \,\, \,\, \,\,\op{mod} \Bigl (H_{*}(B) \otimes \mathbb Q[y] \Bigr)(1-y^2).
\end{align*}
If, moreover,  $E, B, F$ are projective, then  by taking degrees we have:
$$I\chi_y(E) -I\chi_y(F)I\chi_y(B) \equiv \frac{\sigma(E)-\sigma(F)\sigma(B)}{2}(1+y) \,\, \op{mod} \, \, 1-y^2,$$
where $\sigma(-)$ denotes the Goresky-MacPherson intersection cohomology signature.
\begin{align*}
(2) \, \, f_*IT_{y*}(E) -I\chi_y(F)IT_{y*}(B) & \equiv \Bigl (f_*IT_{0*}(E) -I\chi_0(F)IT_{0*}(B) \Bigr)(1-y^2) \\
& \hspace{1cm}+ \frac{f_*IT_{1*}(E)-\sigma(F)IT_{1*}(B)}{2}(y+ y^2) \\
& \hspace{4cm}\op{mod} \Bigl (H_{*}(B) \otimes \mathbb Q[y] \Bigr)(y-y^3).
\end{align*}
If, moreover,  $E, B, F$ are projective, at degree level we have
\begin{align*}
I\chi_y(E) -I\chi_y(F)I\chi_y(B) & \equiv \Bigl (I\chi_0(E) -I\chi_0(F) I\chi_0(B) \Bigr)(1-y^2)\\
& \hspace{2cm} + \frac{\sigma(E)-\sigma(F)\sigma(B)}{2}(y^2+y) \,\, \op{mod} \, \, y-y^3.
\end{align*}
\end{thm}

We conclude with the proof of Theorem \ref{ihmod8} from the introduction. For reader's convenience, we recall its statement:
\begin{thm}
 Let $F\hookrightarrow E \to B$ be a complex algebraic fiber bundle of pure-dimensional complex projective algebraic varieties. Then
 \begin{enumerate}
  \item For any odd integer $y$, $I\chi_y(E)\equiv I\chi_y(F)I\chi_y(B) \op{mod} 4$.
  \item If $y \equiv\; 3 \;mod \;4$, then $I\chi_y(E)\equiv I\chi_y(F)I\chi_y(B) \op{mod} 8$.
  \item If $y \equiv\; 1 \;mod \;4$, then $I\chi_y(E)\equiv I\chi_y(F)I\chi_y(B) \op{mod} 8 \Leftrightarrow 
  \sigma(E) \equiv \sigma(F)\sigma(B) \op{mod} 8 $.
 \end{enumerate}
\end{thm}

\begin{proof}
The proof of this result is formally exactly the same as the one of Theorem \ref{mod8} from \cite{RY}, and it rests on the following steps: 
 \begin{enumerate}
  \item[(a)] We have by (\ref{ichisig}) and Remark \ref{odd} that $I\chi_1(X)=\sigma(X)$ for $X$ pure-dimensional and projective, with  $I\chi_1(X)=\sigma(X)= 0$ for $X$ odd-dimensional. If $X$ is even-dimensional, then 
  \begin{equation}\label{pm1} I\chi_1(X) \equiv \; I\chi_{-1}(X) = I\chi(X) \op{mod} 4, \end{equation} 
  where the congruence follows from the duality result of Theorem \ref{ihduality}.
  \item[(b)] Since $I\chi(X)$ is multiplicative in complex algebraic fiber bundles $F\hookrightarrow E \to B$,  (\ref{pm1}) implies that 
  \begin{equation} \sigma(E) \equiv \sigma(F)\sigma(B) \op{mod} 4.\end{equation}
  \item[(c)] The assertions in the theorem follow now from the degree formula of Theorem \ref{thm6} (1), as in \cite[Theorem 3.11, Theorem 4.1]{RY}.
 \end{enumerate}
\end{proof}

\vspace{0.5cm}

\noindent
{\bf Acknowledgements:} The homological congruence formulae in this paper were presented during the second author's talk at the ``Workshop on Stratified Spaces: Perspectives from Analysis, Geometry and Topology, August 22 - 26", in the ``Focus Program on Topology, Stratified Spaces and Particle Physics, August 8 - 26, 2016, The Fields Institute for Research in Mathematical Sciences". The authors would like to thank the organizers 
of the workshop and the staff of The Fields Institute for the wonderful organization and for a wonderful atmosphere to work. The authors are grateful to J\"org Sch\"urmann and to the anonymous referees for carefully reading the manuscript, and for their valuable comments and constructive suggestions.


\end{document}